\documentclass[11pt]{amsart}
\usepackage{amsmath, amsthm, amssymb, latexsym,url}
\usepackage{hyperref}

\author{J. Vandehey}
\thanks{Email: \href{mailto:vandehey@uga.edu}{\nolinkurl{vandehey@uga.edu}}}
\title[Normality of certain digit expansions]{On the joint normality of certain digit expansions}
\date{\today}

\newtheorem{thm}{Theorem}[section]

\newtheorem{q}[thm]{Question}

\newtheorem{lem}[thm]{Lemma}

\begin{document}

\maketitle

\begin{abstract}
We prove that a point $x$ is normal with respect to an ergodic, number-theoretic transformation $T$ if and only if $x$ is normal with respect to $T^n$ for any $n\ge 1$. This corrects an erroneous proof of Schweiger. Then, using some insights from Schweiger's original proof, we extend these results, showing for example that a number is normal with respect to the regular continued fraction expansion if and only if it is normal with respect to the odd continued fraction expansion.
\end{abstract}

%%%%%%%%%%%%
\section{Introduction}\label{section:introduction}
%%%%%%%%%%%%%%%%%

A number $x$ is said to be normal to base $b$ if each digit string occurs with the same relative frequency as every other string of the same length: for example, in a base $10$ normal number, we would expect the digit string $(101)$ to occur as often as $(974)$.  Equivalently, a number $x$ is base $b$ normal if the sequence numbers $\{b^n x\}_{n=1}^\infty$ is equidistributed modulo $1$.

We know very little about which numbers are normal.  In particular, we do not know of any naturally occuring number---such as $\pi$, $e$, or $\sqrt{2}$---which is normal to any base, although we have constructed numbers that are normal---such as Champernowne's constant or the Copeland-Erd\H{o}s constant.  

A more tractable problem is to understand how the set of numbers normal to one base relates to the set of numbers normal to a different base.  For base $b$ normality, one of the strongest theorems is due to Schmidt \cite{schmidt}:

\begin{thm}\label{thm:schmidt}
Say that $r\sim s$ for two integers $r,s\ge 2$, if there exist positive integers $n$, $m$ such that $r^n =s^m$.  Otherwise, we say $r\not\sim s$.

If $r\sim s$, then all numbers that are base-$r$ normal are base-$s$ normal, and vice-versa.

If $r\not\sim s$, then the set of numbers that are normal base $r$ but not even simply normal\footnote{A number is simply normal if all one-digit strings occur with the same relative frequency.} to base $s$ has the cardinality of the continuum.
\end{thm}

The case $r\sim s$ of Theorem \ref{thm:schmidt} is fairly elementary. If we compare, say, the base $2$ and base $4$ expansions of some number $x$, then the first two digits after the decimal point in the base $2$ expansion completely determine the first digit after the decimal point in the base $4$ expansion, and each following pair of base $2$ digits completely determines the next base $4$ digit, and vice-versa.

However, the case $r \not\sim s$ is quite intricate.  Queff\`{e}lec \cite{quef} provides a survey of various proofs. (Moran and Pollington \cite{MP} have given a slight generalization of Theorem \ref{thm:schmidt} to the case of real, not just integer, basees.)

We wish to consider much more general types of expansions than base-$b$ expansions in this paper.

Consider a bounded space $\Omega\subset \mathbb{R}^m$ with a $\sigma$-algebra $\Sigma$ and a transformation $T: \Omega \to \Omega$ called a number-theoretic transformation. For the purposes of this paper, a transformation is said to be a number-theoretic transformation if the following conditions hold:
\begin{enumerate}
\item We have $T^{-1}\Sigma \subseteq \Sigma$.
\item There exists a set $\mathcal{D}\subset \mathbb{N}$, known as the set of digits, and a partition of $\Omega$ into disjoint sets, $\{I_d\}_{d\in \mathcal{D}}$, such that $\bigcup_{n\in \mathcal{D}} I_n= \Omega$.
\item The restriction of $T$ to $I_d$, denoted by $T_d:=T|_{I_d}$, is continuous and injective.
\item For any finite string of digits $s=[d_1, d_2, \dots, d_l]$, we denote the cylinder set corresponding to this string by
\[
C_s= C[d_1, d_2, \dots, d_l]= T_{d_1}^{-1}T_{d_{2}}^{-1}\dots T_{d_l}^{-1} \Omega \in \Sigma.
\]
We say that a string $s$ is admissible if $C_s$ is non-empty. For all points $x\in \Omega$, there exists a unique infinite string of digits $[\mathfrak{d}_1,\mathfrak{d}_2, \dots]$ such that $x$ is the only point contained in all of the nested sequence of cylinder sets,
\[
C[\mathfrak{d}_1] \supset C[\mathfrak{d}_1, \mathfrak{d}_2] \supset \dots \supset  C[\mathfrak{d}_1, \mathfrak{d}_2, \dots, \mathfrak{d}_n] \supset.
\]
We will abuse notation slightly and often write $x=[\mathfrak{d}_1,\mathfrak{d}_2, \dots]$ and call this the $T$-expansion of $x$. Note that $T$ is a forward shift on 

(For clarification, $d_n$ refers to the $n$th digit of a string $s$ and $\mathfrak{d}_n$ refers to the $n$th digit of a point $x\in \Omega$.)

\item There exists a unique probability measure $\mu:\Sigma \to [0,1]$, which is equivalent to Lebesgue measure, for which $T$ is invariant; that is, $\mu(T^{-1}A)=\mu(A)$ for all $A\subset \Omega$. Moreover, $\mu(\Omega)=1$.

\item $T$ is ergodic; that is, for any  $A\in \Sigma$, if $T^{-1}A$ equals $A$ up to a set of $\mu$-measure zero, then either $\mu(A)=1$ or $\mu(A)=0$.
\end{enumerate}

Given a space $\Omega$ with a number-theoretic transformation $T$, we say that $x\in \Omega$ is $T$-normal if for all admissible strings $s$, we have
\[
\lim_{N\to \infty} \frac{\# \left\{ 0 \le n < N   \mid  T^n x \in C_s  \right\} }{N}= \mu(C_s).
\]
By the pointwise ergodic theorem, almost all points $x\in \Omega$ are $T$-normal. Given two number-theoretic transformations $T$ and $S$ on the same space $\Omega$, we say that $T$ and $S$ are normal-equivalent if a point $x$ is $T$-normal if and only if it is $S$-normal. 

In \cite{schweiger}, Schweiger investigated when $T$ and $S$ are normal-equivalent.  He claimed to show the following result, but one direction of his proof is unfortunately in error, as we will describe in Section \ref{sec:error}:
\begin{thm}\label{thm:schweiger}
If there exist positive integers $n$ and $m$ such that $T^n = S^m$, then $T$ and $S$ are normal-equivalent.
\end{thm}
Since the transformation $Tx = rx \pmod{1}$ is a number-theoretic transformation on $\Omega=[0,1)$, this can be seen as a natural generalization of one direction of Theorem \ref{thm:schmidt}. The fixed proof is in Section \ref{sec:fixedproof} and is in fact quite short.

We quickly remark that if $\mu$ is an ergodic measure corresponding to $T$ and is left invariant by $T$, then $\mu$ is also an ergodic measure corresponding to $T^k$ and is invariant under $T^k$ as well. The ergodicity here is not trivial or self-evident, see \cite{schweiger}.

Schweiger further conjectured that the converse of Theorem \ref{thm:schweiger} was also true---that is, $T$ and $S$ are normal-equivalent if and only if you can write $T^n = S^m$.  Kraaikamp and Nakada \cite{KN1} provide two simple counterexamples to this, utilizing the regular continued fraction, backwards continued fraction, and nearest integer continued fraction.

Although Schweiger's proof is incorrect, it contains an interesting idea that can be used to provide more results not covered by Theorem \ref{thm:schweiger}. In particular, we will investigate what we call augmented transformations $\widetilde{T}$ corresponding to a number-theoretic transformation $T$. The augmented transformation $\widetilde{T}$ encodes extra information about the way the transformation $T$ acts.  (For more details, see section \ref{sec:augment}.) If we construct $\widetilde{T}$ in the right way, then we can compare $\widetilde{T}$ and a different number-theoretic transformation $S$ to show that all $T$-normal numbers are $S$-normal as well.

We will in particular be able to show the following result.
\begin{thm}\label{thm:main}
Let $T_{RCF}$ and $T_{OCF}$ refer to the forward shift on the regular continued fraction digits and odd continued fraction digits, respectively.  Then $T_{RCF}$ and $T_{OCF}$ are normal-equivalent.  
\end{thm} 

This result is closer related to Theorem \ref{thm:schweiger} than it first appears. In particular, there exists a set of RCF cylinders $\{C_s\}$ which cover almost all points in $\Omega=[0,1)$, and functions $n(x)$ and $m(x)$ which are constant on each cylinder in this set, such that $T_{RCF}^{n(x)} = T_{OCF}^{m(x)}$. In other words, just as how certain blocks of digits in base $2$ correspond to individual certain (one-digit) blocks in base $4$, there are certain blocks of RCF and OCF digits that are equivalent to one another. Thus, the following question is quite natural to ask:
\begin{q}
Suppose $T$ and $S$ are number-theoretic transformations on the same space such that there exist a set of cylinders completely partitioning $\Omega$ and functions $n(x)$ and $m(x)$, constant on each of these cylinders, such that $T^{n(x)}=S^{m(x)}$.  Must $T$ and $S$ be normal-equivalent?  Or, if $T$ and $S$ are normal-equivalent, must there exist a corresponding partition into cylinder sets and functions $n(x)$ and $m(x)$?
\end{q}

At the end of the paper we shall briefly outline how $T_{RCF}$-normal numbers must also be $T_{ECF}$-normal, where $T_{ECF}$ is the forward shift on the even continued fraction digits.

In proving Theorems \ref{thm:schweiger} and \ref{thm:main}, we will make use of two techniques that are perhaps not as well known as they should be. The first is the Pyatetski\u\i-Shapiro normality criterion (see \cite{postnikov,PS} for some of the original formulations and \cite{BM,MS,shkredov} for some extensions and improvements).

\begin{thm}[Pyatetski\u\i-Shapiro normality criterion]
Let $T$ be a measure-theoretic transformation on $\Omega$ with ergodic, $T$-invariant measure $\mu$. Let $x\in \Omega$ be a fixed point. If there exists a constant $C\ge 1$ such that for all admissible strings $s$ we have
\[
\limsup_{N\to \infty} \frac{\# \left\{  0 \le n < N \mid T^n x \in C_s   \right\} }{N} \le C \cdot \mu(C_s),
\]
then $x$ is $T$-normal.
\end{thm}

The second technique is the following lemma, whose proof is immediate from the definition of normality and the pointwise ergodic theorem, and which does not appear to be used in the literature.
\begin{lem}\label{lem:subset}
Let $X\subset \Omega$, and suppose for all admissible strings $s$, the limit 
\[
\lim_{N\to \infty} \frac{\# \left\{   0 \le n < N \mid T^n x \in C_s \right\}}{N}, \qquad x\in X
\]
exists and is independent of $x$. If $X$ has positive $\mu$-measure, then $X$ must be a subset of the set of $T$-normal numbers.
\end{lem}

\section{On Theorem \ref{thm:schmidt}}
%%%%%%%%%%%%%%%%%%%%%%%%%%
\subsection{The flaw in Schweiger's argument}\label{sec:error}
%%%%%%%%%%%%%%%%%%%%%%%%%%

Schweiger reduces the problem to showing that $T$-normality is equivalent to $T^k$-normality. It is only the proof that $T$-normality implies $T^k$-normality that is in error.  For simplicity, Schweiger restricts himself to the case $k=2$ in his proof and then states that the method is generalizable.

Schweiger considers a string $s=[d_1, d_2, \dots, d_m]$ with $m$ even, so that $C_s$ is also a cylinder set for $T^2$.  In addition to the cylinder sets, he considers the specialized subsets
\[
E(i) = \{ x \in C_s \mid T^{m+i} x \in C_s\}.
\]
If $T^n x\in E(i)$, then the string $s$ occurs at two places in the $T$-expansion of $x$, starting at the $n+1$th place and at the $n+m+i$th place.

Schweiger goes on to claim without proof that for any $r \in \mathbb{N}$, we have
\[
\#\left\{0\le n < N \mid T^{2n}x \in C_s \right\} \ge \#\left\{0 \le n < N \mid T^n x\in \bigcup_{i=1}^r E(i) \right\}+ o(N).
\]
There are at least two typos in the statement of this inequality and one significant error.  

The first typo is that the range for $n$ on the right hand side should be $0 \le n < 2N$. 

The second typo is that $E(i)$ should be replaced with $E(2i-1)$ to only consider odd indexed sets.  If $i$ is even, and all the occurences of $s$ start at odd indices, then they could be counted by the right-hand side but not the left---a clear contradiction.

Thus it appears that Schweiger had wanted to write
\[
\#\{0\le n < N \mid T^{2n}x \in C_s\} \ge \#
\left\{0 \le n < 2N \mid T^n x\in \bigcup_{i=1}^r E(2i-1) \right\}+ o(N).
\]

However, even this altered statement is incorrect.  We will illustrate with the base-$3$ expansion and $k=2$.  Consider the number
\[
x=0.\overline{20101012012222222222}
\]
and the string $s=[0,1]$.  In this case, $T^{20n+1}(x) \in E(5)$, $T^{20n+3}(x) \in E(3)$, and $T^{20n+5}(x) \in E(1)$, so that the first $2N$ forward iterates of $x$ are in $\bigcup_{i=1,3,5} E(i)$ exactly $3N/10+O(1)$ times; however, by pairing up digits to emulate the base-$9$ expansion, we see
\[
x=0.\overline{(20)(10)(10)(12)(01)(22)(22)(22)(22)(22)},
\]
so that in the first $N$ forward iterates of $x$ (under $T^2$) contain are in $C_s$ exactly $N/10+O(1)$ times.  This is a counter-example to the inequality, and unfortunately this error does not appear to have a simple fix. The proof relies critically on it.

The idea of Schweiger's argument is that $T^n x\in E(2i-1)$ essentially ``reads'' the presence of the desired string in multiple places, and only one of these can also be read by the sped up transformation $T^2$; the problem is that the second appearance of the string $s$ can be read multiple times by different $E(i)$'s.

%%%%%%%%%%%%%%%%%%%%%%%%%%%%%%%%
\subsection{The fixed proof}\label{sec:fixedproof}
%%%%%%%%%%%%%%%%%%%%%%%%%%%%%%%%

Assume $x$ is $T$-normal. Let $s=[d_1, d_2, \dots, d_n]$ be an admissible string in digits of $T^k$. This naturally corresponds to a string $s'=[d'_1, d'_2, \dots, d'_{nk}]$ in digits of $T$, so that $C_s=C_{s'}$. Then 
\begin{align*}
\limsup_{N\to \infty} \frac{\#\{ 1\le n < N \mid (T^k)^n x \in C_s \}}{N} &\le \limsup_{N\to \infty}  \frac{\#\{ 1\le n < kN \mid T^n x \in C_{s'} \}}{N}\\
&= k \cdot \limsup_{N\to \infty} \frac{\#\{ 1\le n < kN \mid T^n x \in C_{s'} \}}{kN}\\
&= k \cdot \lim_{N\to \infty} \frac{\#\{ 1\le n < kN \mid T^n x \in C_{s'} \}}{kN}\\
&= k \cdot \mu(C_{s'}) = k \cdot \mu(C_s) .
\end{align*}
Thus, by the Pyatetski\u\i-Shapiro normality criterion, we have that $x$ is $T^k$-normal as well.

%%%%%%%%%%%%%%%%%%%%%%%%%%%%%%%%%%%%%%
\section{The augmented system}\label{sec:augment}
%%%%%%%%%%%%%%%%%%%%%%%%%%%%%%%%%%%%%%%

We wish to extend a number-theoretic transformation $T$ to a transformation $\widetilde{T}$ on a larger domain $\widetilde{\Omega}$ in a way that will allow us to keep track of certain features of the $T$-expansion.

To create this augmented transformation, we will want a finite set $A$ of natural numbers and consider a new system $(\widetilde{\Omega}, \widetilde{\Sigma},\widetilde{T}, \widetilde{\mathcal{D}}, \widetilde{\mathcal{I}},\widetilde{\mu})$, such that
\begin{itemize}
\item we have $\widetilde{\Omega}=\Omega \times A$, $\widetilde{\Sigma}= \Sigma \times A$, $\widetilde{\mathcal{D}} = \mathcal{D}\times A$, and $\widetilde{\mathcal{I}}=\mathcal{I}\times A$;
\item for $(x;a)\in \widetilde{\Omega}$, with $x\in \Omega$, $a\in A$, we have $\widetilde{T}(x;a)=(Tx;f_x(a))$ for some bijective function $f_x: A \to A$;
\item we have 
\[
\widetilde{\mu}(I \times \{a\}) =\frac{1}{|A|} \mu(I)
\]
for any measurable set $I\subset \Omega$, and that $\widetilde{T}$ is $\widetilde{\mu}$-measure preserving.
\end{itemize}
Given $x\in\Omega$ and $a\in A$, we have $\widetilde{T}^{n-1}(x;a)=(T^{n-1} x; a_n)$ and will refer to $a_n$ as the \emph{augmented value} of the $n$th digit $\mathfrak{d}_n$. (Note that for different initial choices of $a$, we may have different augmented values for the same digit.) We say the augmented value of a string $s=[d_1, d_2, \dots , d_m]$ occuring at the $n$th place of $(x;a) \in \widetilde{\Omega}$ is equal to $a_n$.

Let us use the base-$2$ expansion as a straight-forward example. Here, the transformation $T$ is given by $2x\pmod{1}$.  Given $A=\{1,2,\dots,k\}$, we can augment the base-$2$ digit system with
\[
\widetilde{T}(x;a) = (Tx; a+1\bmod{k}).
\]
  For this system, the augmented value associated to the digits of $(x;1)$ is the mod $k$ value of the place of the digit---that is, $a_n$, the augmented value of $\mathfrak{d}_n$, is $n\pmod{k}$. In this case we see a very clear connection between when $\widetilde{T}(x;1)$ is in the set $\Omega \times \{1\}$ and the iterates $(T^k)^n x$

We are interested in particular augmented transformations, which we will call staggered transformations.  An augmented transformation is said to be staggered if there exists a string $s_{\operatorname{stag}}=[d_1, d_2, \dots , d_n]$ of digits in $\mathcal{D}$ such that for every $a$ and $a'$ there exists a $i<n$ such that
\[
\widetilde{T}^i \left(C_{s_{\operatorname{stag}}} \times \{a'\} \right) \subset \Omega \times \{a\} .
\]
Any string with this property is called a staggered string.\footnote{In an earlier draft of this paper, the staggered strings more closely resembled the sets $E(i)$ from Schweiger's proof, although the author subsequently realized that a simpler definition could be used. This is why it was stated in the abstract that ideas from Schweiger's proof were used to extend the result.}

The example augmented base-$2$ expansion given above is a staggered system. Any string of length $k$ is already a staggered string.

As a quick side note, it is known that, for any admissible string $s=[d_1,d_2,\dots,d_n]$, there exists (by the Radon-Nikodym theorem and the non-singularity of $T$) a $\Sigma$-measurable function $\omega_s:\Omega\to \mathbb{R}$ that satisfies
\[
\lambda(  T|^{-n}_{C_s} E) = \int_E \omega_s \ d\lambda,
\]
where $\lambda$ is the Lebesgue measure on $\Omega$ and $T|^{-n}_{C_s} E$ denotes $(T^{-n}E )\cap C_s$. (See Section 9.2.2 in \cite{schweigerbook}.)

If $\widetilde{T}$ is staggered and the following conditions are also satisfied, then we say that $\widetilde{T}$ is a \emph{staggeringly good} augmented transformation
\begin{itemize}
\item All strings are admissible and all cylinders of $T$ are full---that is, if $s$ has $n$ digits, then $C_s$ is non-empty and, in fact, $T^n C_s = \Omega$, up to some set of $\mu$-measure $0$.
\item The function $f_x$ given by $\widetilde{T}(x;a) = (Tx; f_x(a))$ is the same for all $x$ in a given cylinder $C_s$ with $s=[d_1]$.
\item  The transformation $T$ satisfies Renyi's condition. If $s$ is any admissible string of $n$ digits and $\omega_s$ is defined as above, then there is a absolute constant $\mathcal{C}$, not depending on $s$, such that
\[
\sup_{x\in \Omega} \omega_s(x) \le\mathcal{C} \inf_{x\in\Omega} \omega_s(x).
\]
\end{itemize}
Since $\mu(T|^{-n}_{C_s} E)$ can be written as $\int_E \omega_s(x) \ d\mu(x)$, Renyi's condition implies that for $E\subset \Omega$, we have
\[
\frac{\mu(T|^{-n}_{C_s} E)}{\mu(C_s)} = \frac{\mu(T|^{-n}_{C_s} E)}{\mu(T|^{-n}_{C_s} \Omega)} = \frac{\int_E \omega_s(x) d\mu(x)}{\int_\Omega \omega_s(x) d \mu(x)} \ge \frac{1}{\mathcal{C}}\cdot \frac{\mu(E)}{\mu(\Omega)} .
\]

The importance of staggered systems is the following.

\begin{thm}\label{thm:staggered}
If $\widetilde{T}$ is a staggeringly good augmented transformation on $\widetilde{\Omega}$, then $\widetilde{T}$ is ergodic.  Moreover, if $x$ is $T$-normal then $(x;a)$ is $\widetilde{T}$-normal for any $a\in A$; and conversely.
\end{thm}

\begin{proof}
Suppose $E$ is an invariant, measurable subset of $\widetilde{\Omega}$ with non-zero $\tilde{\mu}$-measure. Then, by projecting onto $\Omega$, we see that the set
\[
\left\{x\in \Omega \mid \text{There exists }a \in A\text{, with }(x;a) \in E\right\}
\]
must have full $\mu$-measure, since it is invariant under $T^{-1}$. This implies $\tilde{\mu}(E)\ge |A|^{-1}$.

Now we apply the ergodic decomposition theorem to $\tilde{\mu}$. Since the only possible invariant sets of non-zero $\tilde{\mu}$-measure on $\widetilde{\Omega}$ have size at least $|A|^{-1}$, this means that there are at most $|A|$ ergodic measures for $\widetilde{T}$, say $\mu_1, \mu_2, \dots, \mu_J$, which are absolutely continuous with respect to $\tilde{\mu}$. Each $\mu_j$ corresponds to a set $E_j \subset \widetilde{\Omega}$ such that $E_j$ is invariant under $\widetilde{T}^{-1}$ and
\[
\mu_j (B) = \frac{\tilde{\mu} (B \cap E_j)}{\tilde{\mu}(E_j)}.
\]
The sets $E_1, E_2, \dots, E_J$ are all distinct.  Hence we have 
\[
\tilde{\mu} = \tilde{\mu}(E_1) \mu_1 +  \tilde{\mu}(E_2) \mu_2 + \dots + \tilde{\mu}(E_J) \mu_J.
\]  

 Let $E_{j,a}$ denote the set of $x\in \Omega$ such that $(x;a) \in E_j$. We claim that 
\[
\mu_j(E_{j,a} \times\{a\})  >0
\]
for all $j$ and all $a\in A$. To show this, let $j$ and $a$ be fixed and let $s_{\operatorname{stag}}$ be a staggered string.  Since $E_j$ must project onto a full $\mu$-measure set in $\Omega$ as described earlier, there must exist at least one $a' \in A$ with 
\[
\mu_j(C_{s_{\operatorname{stag}}} \times \{a'\}) = \frac{1}{\tilde{\mu}(E_j)} \tilde{\mu}\left(\left(C_{s_{\operatorname{stag}}} \times \{a'\} \right)\cap E_j\right) > 0.
\]
By the definition of being a staggered string, however, there exists some $i$ such that 
\[
 C_{s_{\operatorname{stag}}} \times \{a'\} \subset T^{-i}\left(\Omega \times \{a\}\right),
\]
which implies that
\begin{align*}
0&< \tilde{\mu}\left(\left(C_{s_{\operatorname{stag}}} \times \{a'\} \right)\cap E_j\right)  \le \tilde{\mu}\left(T^{-i}\left(\Omega \times \{a\} \right)\cap E_j\right) \\
&= \tilde{\mu}\left(T^{-i}\left((\Omega \times \{a\}) \cap E_j\right)\right) = \tilde{\mu}\left((\Omega \times \{a\} )\cap E_j\right)\\
&= \tilde{\mu} \left( E_{j,a} \times \{a\} \right),
\end{align*}
as desired.

However, we can use Renyi's condition to say even more about the value of $\mu_j$ on cylinders. Say $s=[d_1,d_2,\dots,d_n]$, then 
\[
\widetilde{T}^n\left( C_s \times \{ a \} \right)  = \Omega \times \{a'\}
\]
for some $a' \in A$. (Here we implicitly used the first two conditions needed for $\widetilde{T}$ to be staggeringly good.) Since there are only finitely many $\mu_j$ and only finitely many $a'\in A$, there must exist $\epsilon>0$, such that
\[
\mu_j\left( \Omega \times \{a'\} \right) = \tilde{\mu}(E_{j,a'}) \ge \epsilon
\]
for all $j$ and $a'$. But by applying Renyi's condition, we have that there must exist some $\epsilon' >0$ such that for any $s$ and $a$ we have
\begin{align*}
\mu_j(C_s\times\{a\}) &= \mu_j \left( \widetilde{T}^{-n}|_{C_s\times\{a\}} \left( \Omega \times \{a'\} \right)\right) =\frac{1}{\tilde{\mu}(E_j)} \tilde{\mu} \left( \widetilde{T}^{-n}|_{C_s\times\{a\}} \left( \Omega \times \{a'\} \right) \cap E_j\right)\\
&=\frac{1}{\tilde{\mu}(E_j)} \tilde{\mu} \left( \widetilde{T}^{-n}|_{C_s\times\{a\}} \left(\left( \Omega \times \{a'\}\right) \cap E_j\right)\right)\\
&=\frac{1}{\tilde{\mu}(E_j)} \tilde{\mu} \left( \widetilde{T}^{-n}|_{C_s\times\{a\}} \left(E_{j,a'}\times\{a'\}\right)\right) \\
&= \frac{1}{|A| \tilde{\mu}(E_j)} \mu\left( T|^{-n}_{C_s} E_{j,a'} \right) \ge \frac{\mu (E_{j,a'})}{\mathcal{C} |A| \tilde{\mu}(E_j)} \cdot \mu(C_s)\\
&\ge \epsilon' \cdot \mu(C_s)
\end{align*}
again uniformly over all $j$ and $a$. Therefore, we have
\[
\tilde{\mu}\left(E_j \cap (C_s \times \{a\})\right) \ge \epsilon' |A| \cdot \tilde{\mu}(C_s\times \{a\})
\]
By a martingale convergence--type argument (see pages 50--51 in \cite{schweigerbook}), one can show that any invariant set $E_j$ with non-zero $\tilde{\mu}$-measure must have full $\tilde{\mu}$-measure. Thus $\widetilde{T}$ is ergodic.

Now consider a point $x\in \Omega$ that is $T$-normal. Then for every cylinder $C_s$ and every $a\in A$, we have
\begin{align*}
\lim_{N\to \infty} \frac{1}{N}\#\{ 1\le n \le N \mid \widetilde{T}^n (x;a) \in C_s \times A\} &=\lim_{N\to \infty} \frac{1}{N}\#\{ 1\le n \le N \mid T^n x \in C_s \} \\
&= \mu(C_s) .
\end{align*}
Thus, in particular, we have for any $a' \in A$ and any $j\le J$
\begin{align*}
&\limsup_{N\to \infty} \frac{1}{N}\#\{ 1\le n \le N \mid \widetilde{T}^n (x;a) \in C_s \times \{a'\}\}\\
 &\qquad\le \limsup_{N\to \infty} \frac{1}{N}\#\{ 1\le n \le N \mid \widetilde{T}^n (x;a) \in C_s \times A\}\\
&\qquad= \mu(C_s)\\
&\qquad\le \frac{1}{|A|} \tilde{\mu}\left( C_s \times \{ a \} \right).
\end{align*}
Thus by the Pyatetski\u\i-Shapiro normality criterion, the points $(x;a)$ for \emph{all} $a$ are $\widetilde{T}$-normal with respect to $\tilde{\mu}$.

Alternately, if $(x;a)\in \widetilde{\Omega}$ is $\widetilde{T}$-normal, then 
\begin{align*}
\lim_{N\to \infty} \frac{1}{N}\#\{ 1\le n \le N \mid T^n x \in C_s \} &=\lim_{N\to \infty} \frac{1}{N}\#\{ 1\le n \le N \mid \widetilde{T}^n (x;a) \in C_s \times A\}\\
&= \mu (C_s)
\end{align*}
and hence $x$ is $T$-normal.
\end{proof}

It would be interesting to know if the additional characteristics of being staggeringly good (such as Renyi's condition) are necessary to prove a theorem like the one above.

%%%%%%%%%%%%%
\section{Proof of Theorem \ref{thm:main}}
%%%%%%%%%%%%%

\subsection{Background on continued fractions}

(The details of the RCF and OCF expansions can be found in Masarotto \cite{Masarotto}.)

Let use consider the digit system for the regular continued fraction.  Thus, we have $\Omega=[0,1)\setminus \mathbb{Q}$,
\[
T_{RCF} x =\begin{cases} \dfrac{1}{x}-\left\lfloor \dfrac{1}{x} \right\rfloor & x\neq 0\\
0 & x=0
\end{cases}
\]
$\mathcal{D}=\mathbb{N}$, $I_n=(\frac{1}{n+1},\frac{1}{n}]$ for $n \in \mathcal{D}$, and
\[
\mu_{RCF}(A)=\frac{1}{\log 2}\int_A \frac{1}{1+x} \ dx.
\]
The digits $\alpha_n$ of a given $x\in \Omega$ are given by
\[
\alpha_n= \left\lfloor \frac{1}{T^{n-1}x} \right\rfloor, \quad \text{for } T^{n-1}x \neq 0.
\]
Since we have assumed all $x\in \Omega$ are irrational, this gives
\[
x=\cfrac{1}{\alpha_1+\cfrac{1}{\alpha_2+\cfrac{1}{\alpha_3+\dots}}}=[\alpha_1,\alpha_2, \alpha_3, \dots].
\]

We also want to consider the digit system for continued fraction with odd partial quotients (OCF).  In this case, we have $\Omega=[0,1)\setminus \mathbb{Q}$,
\[
T_{OCF} x = \begin{cases} 
\frac{1}{x}-\left\lfloor \frac{1}{x} \right\rfloor & \text{if }x\neq 0\text{ and }\lfloor1/T_{RCF}x\rfloor \text{ is odd}\\
1-\frac{1}{x}+\left\lfloor \frac{1}{x} \right\rfloor & \text{if }x\neq 0 \text{ and }\lfloor1/T_{RCF}x\rfloor \text{ is even}\\
0 & x=0
\end{cases}
\]
$\mathcal{D}=\left(\mathbb{N}_{odd}\times \{\pm 1\}\right)\setminus \{(1,-1)\}$, 
\[
I_d =\begin{cases} [1/(a+1),1/a), & \text{if }d=(a,+1)\\
[1/a,1/(a-1)), & \text{if }d=(a,-1),
\end{cases}
\]and
\[
\mu_{OCF}(A) = \frac{1}{3\log G} \int_A \frac{1}{G+x-1} + \frac{1}{G+1-x} \ dx,
\]
where $G=(1+\sqrt{5})/2$.  In this case, we have, for irrational $x\in \Omega$,
\[
x=\cfrac{1}{\alpha_1+\cfrac{\epsilon_1}{\alpha_2+\cfrac{\epsilon_2}{\alpha_3+\dots}}}=[(\alpha_1,\epsilon_1),(\alpha_2,\epsilon_2),(\alpha_3,\epsilon_3),\dots].
\]

\subsection{The RCF-to-OCF algorithm}

Consider for a moment a general continued fraction of the form
\[
\cfrac{1}{\alpha_1+\cfrac{\epsilon_1}{\alpha_2+\cfrac{\epsilon_2}{\alpha_3+\dots}}}=[(\alpha_1,\epsilon_1),(\alpha_2,\epsilon_2),(\alpha_3,\epsilon_3),\dots]
\]
where $\alpha_n+\epsilon_n >1$, $\alpha_n \in \mathbb{N}$ and $\epsilon_n=\pm 1$.  We have two operations we can perform on a given continued fraction, which alter the digits but do not change the value of the resulting continued fraction.  The first, is called insertion and is given by
\begin{align*}
&[\dots,(\alpha_{n-1},\epsilon_{n-1}),(\alpha_n,\epsilon_n),(\alpha_{n+1},\epsilon_{n+1}),\dots]\\
&\qquad=[\dots,(\alpha_{n-1},\epsilon_{n-1}),(\alpha_n+\epsilon_n,-\epsilon_n),(1,1),(\alpha_{n+1}-1,\epsilon_{n+1}),\dots].
\end{align*}
The second is called singularization and is given by
\begin{align*}
&[\dots,(\alpha_{n-1},\epsilon_{n-1}),(\alpha_n,\epsilon_n),(1,1),(\alpha_{n+2},\epsilon_{n+2}),\dots]\\
&\qquad=[\dots,(\alpha_{n-1},\epsilon_{n-1}),(\alpha_n+\epsilon_n,-\epsilon_n),(\alpha_{n+2}+1,\epsilon_{n+2}),\dots].
\end{align*}
In both of these cases, we refer to the process as inserting or singularizing at $\mathfrak{d}_n=(\alpha_n,\epsilon_n)$. We can see from the above that the act of inserting cancels out a singularization at the same digit, and vice-versa.

With these two procedures, we have an algorithm that converts the RCF expansion of $x$ into the OCF expansions of $x$.  We write the RCF expansion of $x$ in the more general setting
\[
x=[\mathfrak{d}_1,\mathfrak{d}_2, \dots] = [(\alpha_1, 1),(\alpha_2,1),\dots].
\]
Let $m=1$.  If $\alpha_m$ is even and $\alpha_{m+1}>1$, we insert at $\mathfrak{d}_m$. If $\alpha_m$ is even and $\alpha_{m+1}=1$, we singularize at $\mathfrak{d}_m$.  Then we increase $m$ by $1$ and repeat ad infinitum.

We adopt the convention of referring to what happens in the RCF-to-OCF algorithm in terms of the original RCF expansion: let us illustrate this now with an example. Consider the simple RCF expansion
\[
[4,3,3,3,\dots] = [(4,1), (3,1), (3,1), (3,1), \dots].
\]
In the RCF-to-OCF algorithm, the first step would be to insert at the first digit, $\frak{d}_1=(4,1)$ and obtain:
\[
[(5,-1), (1,1), (2,1), (3,1), (3,1),\dots].
\]
The next step would be to insert at the now-third digit $(2,1)$. However, this digit naturally arises from the second digit $\frak{d}_2 = (3,1)$ in our original RCF expansion, and therefore we shall refer to this as inserting at $\frak{d}_2$ rather than as inserting at the third digit. We will also say that this is the point when the algorithm reaches $\frak{d}_2$, and that it has been changed to the digit $(2,1)$. Likewise after the next step when we have
\[
[(5,-1), (1,1), (3,-1), (1,1), (2,1), (3,1),\dots]
\]
we will say that the algorithm has reached $\frak{d}_3$ and refer to the next operation as inserting at $\frak{d}_3$, even though the digit $(2,1)$ appears in the fifth place. Given $n$, we let $m(n)$ denote the value of $m$ when the algorithm arrives at $\frak{d}_n$. Note that unless $\frak{d}_n=(1,1)$ and the algorithm singularized at $\frak{d}_{n-1}$, the function $m(n)$ will exist. If this were to happen, we will define $m(n)$ by $m(n-1)$, which must exist.

If a singularization occurs at $\frak{d}_n$, then we say that a deletion occurs at $\frak{d}_{n+1}$.

The RCF-to-OCF algorithm produces a domino-like effect.  It runs along a string of odd $\alpha_m$ without doing anything, until it reaches an even $\alpha_m$.  By inserting or singularizing, we change the parity of $\alpha_m$, making it odd and also changing the parity of the successive $\alpha_{m+1}$ (or, if $\alpha_{m+1}=1$, it deletes that term entirely and instead changes the parity of $\alpha_{m+2}$).  If $\alpha_{m+1}$ was odd, it is now even, so we insert or singularize at it, and thereby alter the parity of the successive term $\alpha_{m+2}$.  This continues until the parity of the successive term was changed from even to odd, thus stopping the domino effect and allowing the algorithm to skip forward over odd $\alpha_m$'s once again.

Thus, all the odd $\alpha_m$'s that occur between the first and second appearances of even $\alpha_m$'s and third and fourth appearances of even $\alpha_m$'s, etc. will be altered (or simply removed) by the algorithm, and all the remaining odd $\alpha_m$'s are left unchanged.  Likewise, the first, third, fifth, etc. even $\alpha_m$'s will always be increased by $1$, while the remaining even $\alpha_m$'s may be increased or decreased depending on what precedes them.  

The challenge of trying to compare RCF-normality with OCF-normality comes from this disjointed nature of this algorithm: whether an odd $\alpha_m$ remains unchanged or gets altered by the algorithm depends on, at least, how many even $\alpha_m$'s precede it.  To circumvent this problem, we use a particular augmentation of the RCF system to keep track of how many even $\alpha_m$'s have passed.

\subsection{Forward knowledge of digits}\label{sec:forward}
To justify the augmented system we will give in the next section, let us pose a slightly different problem. Suppose that we know all the RCF digits of a real number $x$ from the $n$th digit onward, and in addition, we know whether an insertion, singularization, or deletion occured at the $n-1$th digit in the RCF-to-OCF algorithm. From this information, what digits of the OCF expansion of $x$ do we know?

Say $\alpha_n$ is even and no insertion or deletion occured at $\frak{d}_{n-1}$. Then the RCF-to-OCF can start at $\frak{d}_n$ as if it were the first digit. Therefore, in this case, we know all the digits of the OCF expansion of $x$ from the $m(n)$th digit forward.

Say $\alpha_n$ is even and an insertion or deletion occured at $\frak{d}_{n-1}$. Regardless of which occured, when the RCF-to-OCF algorithm reaches $\frak{d}_n$, it will be odd, and thus, no insertion or deletion will occur at $\frak{d}_n$. Although we may not know what the $m(n)$th digit of the OCF expansion of $x$ will be, we will know all the digits from the $m(n)+1$st position onward.

Say $\alpha_n$ is odd and greater than $1$ and no insertion or deletion occured at $\frak{d}_{n-1}$. Then in this case no insertion or deletion will occur at $\frak{d}_n$ and thus we know all the digits of the OCF expansion from the $m(n)$th onward.

Say $\alpha_n$ is odd and greater than $1$ and an insertion or deletion occured at $\frak{d}_{n-1}$. This will change the parity of $\alpha_n$ to even when the algorithm reaches it, and so an insertion or deletion will occur at $\frak{d}_n$ (depending on what the value of $\alpha_{n+1}$ is). Thus, we will know all the digits of the OCF expansion from the $m(n)+1$th digit onward.

If $\alpha_n=1$ and no singularization or deletion occured at $\frak{d}_{n-1}$, then we know the OCF algorithm from the $m(n)$th digit onward.

If, on the other hand, $\alpha_n=1$ and a singularization or deletion occured at $\frak{d}_{n-1}$, then the situation is more delicate. Without knowing whether it was a singularization or a deletion that occured at $\frak{d}_{n-1}$, we cannot immediately say how it will impact the RCF-to-OCF algorithm. Let $n'$ be the smallest positive integer greater than $n$ such that $\alpha_{n'}\neq 1$, then when the algotirhm reaches $n'$, we know that this is a number greater than $1$ with an insertion or deletion occuring at $\frak{d}_{n'-1}$, thus we know what happens from the $m(n')$th or $m(n')+1$ digit onward (depending on whether $\alpha_{n'}$ is even or odd) but not what happens prior to that point.

\subsection{The augmented system}

We consider the following augmentation of the RCF expansion: we let $A=\{1,2\}$ and define $\widetilde{T}$ by
\[
\widetilde{T}(x;a) =\begin{cases}
(T_{RCF}x; a) & \alpha_1(x)\text{ is odd}\\
(T_{RCF}x; 3-a) & \alpha_1(x)\text{ is even}.
\end{cases}
\]
Here $\alpha_1(x)$ refers to the first RCF digit of $x$.  This transformation $\widetilde{T}$ is easily seen to be $\tilde{\mu}_{RCF}$-measure-preserving.  We call this augmented system RCF*.

One can easily show that $\widetilde{T}$ is staggeringly good. The string $s=[2]$ is staggered, and the remaining conditions follow from standard facts about the RCF algorithm. 

The extra information that $\widetilde{T}$ carries is the following: the augmented value of $\mathfrak{d}_n$ in $(x,1)$ is $1$ if in the RCF-to-OCF algorithm no singularization, insertion, or deletion occured at $\frak{d}_{n-1}$, and the augmented value is $2$ otherwise.

\begin{lem}\label{lem:mn}
Suppose that $x$ is RCF-normal. Then there exists a constant $c$ such that $m(n)=cn(1+o(1))$.
\end{lem}

\begin{proof}
We know that $m(n)$ equals $n$ minus the number of singularizations that occured before the $n$th digit, plus the number of insertions that have occured before the $n$th digit, plus $O(1)$. Thus it suffices to show that the number of singularizations that occured before the $n$th digit is $c_1 n(1+o(1))$ and that the number of insertions is $c_2n(1+o(1))$.

Let's start with insertion first. Each insertion occurs to a unique occurence of $\widetilde{T}^i (x;1)$ in a cylinder set $C_s\times \{a\}$ corresponding to one of the following strings
\[
([2a,1^{2(b-1)},c+1],1) \qquad \text{ or } \qquad ([2a+1, 1^{2(b-1)},c+1],2)
\]
where $a,b,c\in \mathbb{N}$. (Here we use the notation $[1^j]$ to denote the string composed of $j$ copies of $1$.) As a simple example, if we encounter the strings $([2,1,1,5],1)$, then we would singularize at the $2$, then insert at the second $1$.

Let $\mathcal{D}$ denote the union over all these cylinders, $\mathcal{D}_k$ denote the union over all of these cylinder sets with $|s|\le k$, and $\mathcal{I}_k = [1^{2k}]$. Then we have that the number of insertions  up to the $n$th place is greater than
\[
\#\{i \le n :\widetilde{T}^i (x;1) \in \mathcal{D}_k\} +O(k)
\]
and less than
\[
\#\{i \le n :\widetilde{T}^i (x;1) \in \mathcal{D}_k\cup \mathcal{I}_k\} +O(k).
\]
Thus, since the RCF-normality of $x$ implies the RCF*-normality of $(x;1)$ by Theorem \ref{thm:staggered}, we have that the number of insertions up to the $n$th place is
\[
n(\tilde{\mu}(\mathcal{D}_k)+O(\tilde{\mu}(\mathcal{I}_k))+o(1)).
\]
By letting $k$ tend to infinity and noting that the measure of $\mathcal{D}\setminus \mathcal{D}_k$ and $\mathcal{I}_k$ tend to $0$ with $k$, we get that the number of insertions is $\tilde{\mu}(\mathcal{D})n(1+o(1))$, as desired.

The proof is similar for singularization, with a few key differences. In this case, each occurence of $\widetilde{T}^i (x;1)$ in a cylinder set $C_s\times \{a\}$ of one of the following strings
\begin{align*}
&([2a,1^{2b-1},c+1],1),\quad ([2a,1^{2b},c+1],1),\\
& \quad ([2a+1, 1^{2b-1},c+1],2)\quad  \text{ or }\quad  ([2a+1, 1^{2b},c+1],2)
\end{align*}
where $a,b,c\in \mathbb{N}$ now corresponds to $b$ total singularizations. So let $\mathcal{E}_k$ denote the union over all these cylinder sets with $|s|=2k$ or $|s|=2k+1$ (i.e., these are the strings with $b=k$), and let $\mathcal{I}_k=[1^{2k}]$. Then by a similar argument to the above, for any fixed $K$, the number of singularizations up to the $n$th place is
\[
n\left( \sum_{k=1}^K k \cdot \tilde{\mu}(\mathcal{E}_k)+O\left( \tilde{\mu}(\mathcal{I}_K)\right) + o(1)\right).
\] 
By noting that the measure of $\mathcal{I}_K$ goes to $0$ as $K$ goes to infinity, the sum in the above equation converges as $K$ goes to infinity. Therefore, we can find a constant $c_1$ so that the number of singularizations up to the $n$th place is $c_1 n(1+o(1))$ as desired.
\end{proof}

Now we want to consider a point $x$ that is RCF-normal. By Lemma \ref{lem:subset}, to prove that all such $x$ are OCF-normal, it suffices to show that for all finite-length OCF strings $s_O$, we have that the limit
\[
\lim_{N\to \infty} \frac{1}{N} \#\{0\le n < N \mid T_{OCF}^{n}x \in C_{s_O}\}
\]
exists and is independent of which RCF-normal point $x$ we used.

Given a point $x$, let $m^{-1}(N)$ denote the smallest positve integer $n$ such that $m(n) \ge N$.

Consider a string $s_O$ of OCF digits, and let $(s;a)$, consisting of a finite RCF string $s$ together with an augmented value $a\in \{1,2\}$, be called \emph{a trigger string} for $s_O$ if the following hold:
\begin{itemize}
\item If $\widetilde{T}^i(x;1)\in C_s \times \{a\}$ for some non-negative integer $i$, then there exists an $j$ such that $T_{OCF}^j x \in C_{s_O}$.
\item If $s'$ is a substring of $s$ with corresponding augmented value $a'$---i.e., if there exists an $n$ such that 
\[
\widetilde{T}^n \left( C_s \times \{a\} \right)\subset C_{s'} \times \{a'\}
\]
---then $(s';a)$ does not satisfy the previous condition.
\end{itemize}

We can actually give $j$ fairly explicitly in terms of $i$, since the minimality criterion of the second condition tells us that the $i+|s|-1$th digit of the RCF expansion of $x$ should roughly coincide with the $j+|s_O|-1$th digit of the OCF expansion of $x$. If the last digit of $s$ is a $1$ that would be deleted, then $m(i+|s|-2)=j+|s_O|-1$. If the last digit of $s_O$ is a $(1,1)$ that was caused by an insertion at the next-to-last digit of $s$, then $m(i+|s|-2)+1=j+|s_O|-1$.  In all other cases, we have that $m(i+|s|-1)=j+|s_O|-1$.  From this we can see that each occurence of the trigger strings forms a bijection with each appearance of $s_O$ in the OCF expansion.

Let $k$ be a positive integer. Any trigger string of $s_O$ of length at least $2|s_O|+k+4$ will be of the form
\[
([2a,1^k,*];1) \qquad \text{or} \qquad ([2a+1,1^k,*];2)
\]
where $*$ represents some string of digits. This follows from the previous paragraph combined with Section \ref{sec:forward}. (The $2|s_O|$ accounts for the possibility that as many singularizations occur as possible.) Suppose that a trigger string of length $2|s_O|+k+4$ started with $([2a,1^b,c+1,*];1)$ or $([2a+1,1^b,c+1,*];2)$ for $a,b,c\in \mathbb{N}$ and $b<k$: then by Section \ref{sec:forward}, the corresponding string starting at $c+1$ (which is the $i+b+1$th digit) will tell us some of the digits of the OCF expansion starting from at least the $m(i+b+1)+1$th digit onward. But since, by construction, $m(i+b+1)+1<j$, we have that the corresponding string starting from $c+1$ is also a trigger string, thus contradicting the second condition of being a trigger string. A similar argument holds for strings of the form $([1,*];*)$.

\iffalse
As an example, consider the OCF string $[(3,1)]$. There are three classes of trigger strings for $[(3,1)]$. The first class consists of those where the $3$ is unchanged from anything that comes before it. The only trigger string of this clas is $([3];1)$. The second class of strings are those where the $3$ appears due to a singularization right before it. These trigger strings are
\[
([2a,1^{2b-1},2];1)\qquad ([2a+1,1^{2b-1},2];2)
\]
with $a,b\in \mathbb{N}$. The final class are those where the $3$ appears due to an insertion right before it. These trigger strings are 
\[
([2a,1^{2b-2},4];1)\qquad ([2a+1,1^{2b-2},4];2)
\]
again with $a,b\in\mathbb{N}$.

\begin{proof}[Proof of Lemma \ref{lem:stringstart}]
The key idea is that if a trigger string is sufficiently long, the actual value of the digits at the beginning of the string are not important compared with the parity. Suppose a string of the form
\[
([2a,1^b,c+1,*];1)
\]
is a trigger string for $s_O$, with $a,b,c\in \mathbb{N}$, $b< k$. Depending on whether $b$ is even or odd, when the RCF-to-OCF algorithm reaches $c+1$, it will have changed to either a $c$ or $c+2$, regardless we know that it's parity has changed. However, if this string has at least $2|s_O|+k+2$ digits, then the digit corresponding to $c+1$ is not in the string $s_O$, then the string $([c+1,*];2)$ is a trigger string, contradicting that our original string was a trigger string. 

A similar argument works for all the other strings.
\end{proof}
\fi

Let $\mathcal{D}_k(s_O)$ denote the union of all cylinder strings of $s_O$ with length at most $k$. Let $\mathcal{I}_k=[1^{k'}]$ where $k'=2|s_O|-2-k$. Then note that we have
\begin{align*}
&\#\{0\le n \le m^{-1}(N)\mid \widetilde{T}^n (x;1)\in \mathcal{D}_k(s_O)\}+O(k)\\ 
&\qquad\le \#\{0 \le n \le N \mid T_{OCF}^n x \in C_{s_O}\} \\
&\qquad \le \#\{0\le n \le m^{-1}(N)\mid \widetilde{T}^n (x;1)\in \mathcal{D}_k(s_O)\cup \mathcal{I}_k\}+O(k).
\end{align*}
However, by the $\widetilde{T}$-normality of $(x;1)$, we have
\begin{align*}
\tilde{\mu}(\mathcal{D}_k(s_O))m^{-1}(N)(1+o(1))&\le \#\{0 \le n \le N \mid T_{OCF}^n x \in C_{s_O}\}  \\&\le \tilde{\mu}(\mathcal{D}_k(s_O)\cup \mathcal{I}_k)m^{-1}(N)(1+o(1)).
\end{align*}
But as $k$ tends to infinity, the measure of $\mathcal{I}_k$ tends to $0$, and so there exists a constant $c_3$, dependent only on $s_O$, so that
\[
 \#\{0 \le n \le N \mid T_{OCF}^n x \in C_{s_O}\}  = c_3 m^{-1}(N)(1+o(1)).
\]
But by Lemma \ref{lem:mn}, we have that $m^{-1}(N) =c_4 N (1+o(1))$. Thus
\[
\lim_{N\to \infty} \frac{ \#\{0 \le n \le N \mid T_{OCF}^n x \in C_{s_O}\}  }{N} = c(s_O)(1+o(1))
\]
for some constant $c(s_O)$ depending only on $s_O$. This completes the proof in this direction.

\subsection{The reverse direction}

The proof that all OCF-normal numbers are also RCF-normal follows by a similar method. In fact, it is even easier and we can skip RCF* completely. This is because the OCF-to-RCF algorithm is simply to singularize or insert at any digit $\mathfrak{d}_n$ with $\epsilon_n=-1$. Therefore, given an RCF-string $s_R$, it is very easy to find the corresponding OCF trigger strings and apply the same methods as above. We omit the details as they are more tedious than insightful.

%%%%%%%%%%%%%%%%%%%%
\section{RCF normality and ECF normality}
%%%%%%%%%%%%%%%%%%%%%%%

We shall only outline how RCF-normality implies ECF-normality, remarking in how the proof differs from that of the proof that RCF-normality implies OCF-normality.

One major difference is that the invariant measure for the ECF expansion, $\mu_{ECF}$,  is not finite; therefore, for a point $x\in [0,1)$ to be ECF-normal, we mean that given two strings $s$ and $s'$ with $\mu_{ECF}$-finite cylinder sets $C_s$ and $C_{s'}$, we have
\begin{equation}\label{eq:ECFnormal}
\lim_{N\to \infty} \frac{\#\{0 \le n < N \mid T_{ECF}^n x \in C_s\}}{\#\{0 \le n < N \mid T_{ECF}^n x \in C_{s'}\}} = \frac{\mu_{ECF}(C_s)}{\mu_{ECF}(C_{s'})}.
\end{equation}
A variant of Lemma \ref{lem:subset} holds with this new notion of normality. Any subset $X$ of $[0,1)$ of positive lebesgue measure for which the limit on the left in \eqref{eq:ECFnormal} exists and is independent of $x\in X$ is a subset of the set of ECF-normal numbers.

As with the OCF case, there is an RCF-to-ECF algorithm and we provide it in lieu of detailing $T_{ECF}$ and $\mu_{ECF}$. As before, consider a general RCF-expansion of a number $x$
\[
x=[(\alpha_1,\epsilon_1),(\alpha_2,\epsilon_2),\dots]
\]
and let $m =1$.  Suppose $\alpha_m$ is odd. If, in addition, $\alpha_{m+1}=1$, then we singularize at $\alpha_m$; otherwise, we insert at $\alpha_m$.  We then increase $m$ by $1$ and repeat.

This algorithm can be sped up significantly.  The process of insertion introduces a digit $(1,1)$, which, since the corresponding $\alpha_m$ is odd, must be the target of the next insertion or singularization.  Thus, if we are currently considering $\alpha_m$ odd, and $\alpha_{m+1}>1$, then we can replace
\[
[\dots,(\alpha_m,1),(\alpha_{m+1},1),(\alpha_{m+2},1),\dots]
\]
with
\[
[\dots,(\alpha_m+1,-1),(2,-1)^{\alpha_{m+1}-1},(\alpha_{m+2}+1,1),\dots].
\]

As with the RCF to OCF algorithm, the RCF to ECF algorithm produces a domino like effect.  It skips over long periods of even $\alpha_m$'s until it reaches an odd $\alpha_m$.  Regardless of the value of $\alpha_{m+1}$, it is wholly altered, either into nothing (if $\alpha_{m+1}=1$) or into a long string of $(2,-1)$'s (otherwise); then the parity of $\alpha_{m+2}$ is switched from even to odd or vice-versa.  Thus, after first seeing an odd $\alpha_m$, it will continue fundamentally altering the $\alpha_{m+1+2i}$'s and shifting the parity of the $\alpha_{m+2i}$'s until it reaches $\alpha_{m+2i}$ that started odd (and would thus be shifted to even by the insertion or deletion at $\alpha_{m+2i-1}$).

We use the following augmentation of the natural extension of the RCF system: we let $A=\{1,2,3\}$ and define $\widetilde{T}$ by
\[
\widetilde{T}(x;a) =\begin{cases}
(T_{RCF}x; 1) & \text{if }a=1\text{ and }\alpha_1(x)\text{ is even}\\
(T_{RCF}x; 3) & \text{if }a=1\text{ and }\alpha_1(x)\text{ is odd}\\
(T_{RCF}x; 2) & \text{if }a=3\\
(T_{RCF}x; 3) & \text{if }a=2\text{ and }\alpha_1(x)\text{ is even}\\
(T_{RCF}x; 1) & \text{if }a=2\text{ and }\alpha_1(x)\text{ is odd}
\end{cases}
\]
Again, this is easily seen to be staggeringly good, with $[3,3]$ as a staggered string.  Those $\mathfrak{d}_n$ with augmented value $1$ are left unchanged by the algorithm up to the point where $m= n$; those $\mathfrak{d}_n$ with augmented value $2$ are altered by a prior insertion or singularization when at the point where $m= n$; and those $\mathfrak{d}_n$ with augmented value $3$ are either removed completely or are replaced by a string of $(2,-1)$'s.

If we let $m(n)$ have an analogous meaning here as it did in the previous section, then we get a much stronger result than Section \ref{sec:forward}. In fact, $\widetilde{T}^n(x;1)$ always contains enough information to know all the ECF digits of $x$ from the $m(n)+1$th place onward.  So given an ECF string $s_E$, all the trigger strings for $s_E$ will have finite length (never worse than $3|s_E|$). 

If we let $\mathcal{D}(s_E)$ be the corresponding union over all trigger strings of $s_E$ then we have
\begin{align*}
 \frac{\#\{0 \le n < N \mid T_{ECF}^n x \in C_s\}}{\#\{0 \le n < N \mid T_{ECF}^n x \in C_{s'}\}}&=\frac{\#\{0 \le n < m^{-1}(N) \mid  \widetilde{T}^n(x;1) \in\mathcal{D}(s)\}+O(|s|)}{\#\{0 \le n < m^{-1}(N) \mid \widetilde{T}^n(x;1) \in \mathcal{D}(s')\}+O(|s'|)}\\
&= \frac{\tilde{\mu}(\mathcal{D}(s))m^{-1}(N)(1+o(1))}{\tilde{\mu}(\mathcal{D}(s'))m^{-1}(N)(1+o(1))}\\
&= \frac{\tilde{\mu}(\mathcal{D}(s))(1+o(1))}{\tilde{\mu}(\mathcal{D}(s'))(1+o(1))}
\end{align*}
as desired. (We note that $m^{-1}(N)$ is not in any way asymptotic to $N$ in this case, but this is not needed with this different notion of normality.)

\end{document}